\documentclass{amsart}
\usepackage{amssymb}
\usepackage{amsmath}
\usepackage{amsthm}
\usepackage{graphicx}
\usepackage{float}
\usepackage{color}
\usepackage{csquotes}

\newcommand{\re}{\mathop{\mathrm{Re}}}
\newcommand{\im}{\mathop{\mathrm{Im}}}

\newcommand{\artanh}{\mathop{\mathrm{artanh}}}

\newtheorem{theorem}{Theorem}
\newtheorem{lemma}{Lemma}
\newtheorem{example}{Example}
\newtheorem{proposition}{Proposition}
\newtheorem{remark}{Remark}

\title[]{A Strengthening of the Harnack Inequality}

\author{Marek Svetlik}
\address{Faculty of mathematics, University of Belgrade, Studentski Trg 16,
Belgrade, Republic of Serbia} \email{svetlik@matf.bg.ac.rs}

\date{January 16, 2025}



\begin{document}

\begin{abstract}
We prove the stronger version of Harnack's inequality for positive harmonic functions defined on the unit disc.
\end{abstract}

\maketitle

\section{Introduction and the main result}

Let $\mathbb{U}=\{z\in\mathbb{C}:|z|<1\}$ be the unit disc. It is known (see e.g. {\cite[p. 47]{axler}} or {\cite[p. 13]{duren}}) that the harmonic function $u:\mathbb{U}\rightarrow (0,+\infty) $ satisfies Harnack's inequality
\begin{equation}\label{HarnackIneqTheoremIneq}
 \frac{1-|z|}{1+|z|}\leqslant \frac{u(z)}{u(0)}\leqslant \frac{1+|z|}{1-|z|}
\end{equation}
for all $z\in\mathbb{U}$.

In this paper we prove the following stronger version of this classical result.

\begin{theorem}\label{MainTheorem}
Let $u:\mathbb{U}\rightarrow (0,+\infty) $ be a harmonic function. Then
\begin{equation}\label{HarnackIneqTheoremSvIneq}
 \left(\frac{1+|z|^2}{1-|z|^2}+\frac{|\nabla u(0)|}{u(0)}\frac{|z|}{1-|z|^2}\right)^{-1}\leqslant \frac{u(z)}{u(0)} \leqslant\left(\frac{1+|z|^2}{1-|z|^2}+\frac{|\nabla u(0)|}{u(0)}\frac{|z|}{1-|z|^2}\right),
\end{equation}
for all $z\in\mathbb{U}$. Here $\nabla u=(u_{x},u_{y})$ is the gradient of the function $u$.
\end{theorem}

Since for harmonic function $u:\mathbb{U}\rightarrow (0,+\infty) $ we have $|\nabla u(0)|\leqslant 2u(0)$ (see inequality \eqref{Thm:SchwarzPickForPositiveHarmonic:fla1} below), it follows that
\begin{equation}\label{glava3:HarnackIneqSvRemarkI1}
 \frac{1+|z|^2}{1-|z|^2}+\frac{|\nabla u(0)|}{u(0)}\frac{|z|}{1-|z|^2}\leqslant\frac{1+|z|}{1-|z|}
\end{equation}
and
\begin{equation}\label{glava3:HarnackIneqSvRemarkI2}
 \left(\frac{1+|z|^2}{1-|z|^2}+\frac{|\nabla u(0)|}{u(0)}\frac{|z|}{1-|z|^2}\right)^{-1}\geqslant\frac{1-|z|}{1+|z|}
\end{equation}
for all $z\in\mathbb{U}$. Therefore, the inequality \eqref{HarnackIneqTheoremSvIneq} is stronger than the inequality \eqref{HarnackIneqTheoremIneq}.

Fix $c\in \mathbb{U}$ and denote by $\varphi_{c}$ the mapping given by
\begin{equation*}
 \varphi_{c}(z)=\frac{z+c}{1+\overline{c}z}.
\end{equation*}
It is known that $\varphi_{c}$ is a conformal isomorphism of the unit disc $\mathbb{U}$ onto itself. In particular, if $c=1$, then we define $\varphi_c(z)=1$ for all $z\in \mathbb{U}$.

The following example shows that there is function $u:\mathbb{U}\rightarrow(0,+\infty)$ and $z\in \mathbb{U}$, so that the equality in \eqref{HarnackIneqTheoremIneq} holds.

\begin{example}
Let $0\leqslant c\leqslant1$ and let $u_1,u_2:\mathbb{U}\rightarrow(0,+\infty)$ be defined by
\begin{equation*}
    u_1(z)=\re\frac{1+z\varphi_{c}(z)}{1-z\varphi_{c}(z)}.
\end{equation*}
and
\begin{equation*}
    u_2(z)=\re\frac{1-z\varphi_{c}(z)}{1+z\varphi_{c}(z)}.
\end{equation*}

Then $u_1(0)=u_2(0)=1$ and $|\nabla u_1(0)|=|\nabla u_2(0)|=2c$. Further,
\begin{equation*}
    \frac{u_1(x)}{u_1(0)}=\frac{1+x^2}{1-x^2}+\frac{|\nabla u_1(0)|}{u_1(0)}\frac{x}{1-x^2}
\end{equation*}
and
\begin{equation*}
    \frac{u_2(x)}{u_2(0)}=\left(\frac{1+x^2}{1-x^2}+\frac{|\nabla u_2(0)|}{u_2(0)}\frac{x}{1-x^2}\right)^{-1},
\end{equation*}
for all $x\in[0,1)$.
\end{example}

To prove the Theorem \ref{MainTheorem}, we use the hyperbolic densities on the unit disc $\mathbb{U}$ and on the right half-plane $\mathbb{K}=\{z\in\mathbb{C}:\re{z}>0\}$.

The hyperbolic density on the unit disc $\mathbb{U}$ is defined by
\begin{equation*}
 \rho_{\mathbb{U}}(z)=\frac{2}{1-|z|^2}, \quad z\in \mathbb{U}.
\end{equation*}
Based on Riemann's Mapping Theorem we can define the hyperbolic density on an arbitrary simply connected domain $\Omega\subset \mathbb{C}$, $\Omega$ different from $\mathbb{C}$ (we call these domains hyperbolic), in the following way
\begin{equation*}
 \rho_{\Omega}(z)=\rho_{\mathbb{U}}(\phi(z))|\phi'(z)|, \quad z\in\Omega,
\end{equation*}
where $\phi:\Omega\rightarrow\mathbb{U}$ is a conformal isomorphism from $\Omega$ onto $\mathbb{U}$. In particular, if $\psi$ is defined by $\displaystyle \psi(z)=\frac{z-1}{z+1}$, then $\psi$ is a conformal isomorphism from $\mathbb{K}$ onto $\mathbb{U}$ and by direct calculation we obtain
\begin{equation*}
 \rho_{\mathbb{K}}(z)=\frac{1}{\re{z}}, \quad z\in\mathbb{K}.
\end{equation*}

We define the hyperbolic distance on a hyperbolic domain $\Omega$ as follows
\begin{equation}
 d_{\Omega}(z_1,z_2)=\inf\int_{\gamma}\rho_{\Omega}(z)|dz|,
\end{equation}
where the infimum is taken over all $C^1$ curves $\gamma$ that connect $z_1$ with $z_2$ in $\Omega$. Specifically, it can be shown that
\begin{equation*}
 d_{\mathbb{U}}(z,0)=\log{\frac{1+|z|}{1-|z|}}=2\artanh{|z|}.
\end{equation*}

Let $f:\Omega_1\rightarrow\Omega_2$ ($\Omega_1$ and $\Omega_2$ are hyperbolic domains in $\mathbb{C}$) be a holomorphic function. The famous Schwarz-Pick lemma asserts that
\begin{equation}\label{SchwarzPickInequality1}
 \rho_{\Omega_2}(f(z))|f'(z)|\leqslant\rho_{\Omega_1}(z), \quad z\in\Omega_1.
\end{equation}
As a corollary of \eqref{SchwarzPickInequality1} we have
\begin{equation}\label{SchwarzPickInequality2}
 d_{\Omega_2}(f(z_1),f(z_2))\leqslant d_{\Omega_1}(z_1,z_2), \quad \mbox{for all}\quad z_1,z_2\in\Omega_1.
\end{equation}
Moreover, if $f$ is a conformal isomorphism from $\Omega_1$ onto $\Omega_2$, then equalities hold in \eqref{SchwarzPickInequality1} and \eqref{SchwarzPickInequality2}. Thus, if $f$ is a conformal isomorphism from $\Omega_1$ onto $\Omega_2$, then $f$ is an isometry from the metric space $(\Omega_1,d_{\Omega_1})$ onto the metric space $(\Omega_2,d_{\Omega_2})$.

The hyperbolic derivative of $f$ at $z\in\Omega_1$ (for motivation and details see section 5 in \cite{BeardonMinda}, cf. \cite{BeardonCarne}) is defined as follows
\begin{equation*}
 f^h(z)=\frac{\rho_{\Omega_2}(f(z))}{\rho_{\Omega_1}(z)}f'(z).
\end{equation*}

From \eqref{SchwarzPickInequality1} it follows immediately that $|f^h(z)|\leqslant 1$ for all $z\in \Omega_1$.

Using this notion, A. F. Beardon and T. K. Carne proved the following theorem, which is a strengthening of the Schwarz-Pick inequality \eqref{SchwarzPickInequality2}.

\begin{theorem}[\cite{BeardonCarne}]\label{Thm:BeardonCarne}
Let $f:\Omega_1\rightarrow\Omega_2$ ($\Omega_1$ and $\Omega_2$ are hyperbolic domains in $\mathbb{C}$) be a holomorphic function. Then
\begin{equation}\label{BeardonCarne:Inequality}
 d_{\Omega_2}(f(z_1),f(z_2))\leqslant\log\Big(\cosh {d_{\Omega_1}(z_1,z_2)}+|f^h(z)|\sinh{d_{\Omega_1}(z_1,z_2)}\Big).
\end{equation}
\end{theorem}

Note that in the paper \cite{BeardonCarne} the previous theorem was formulated and proved for the case $\Omega_1=\Omega_2=\mathbb{U}$. The version of this theorem that we have stated follows directly from the theorem given in \cite{BeardonCarne}. The inequality \eqref{BeardonCarne:Inequality} plays a crucial role in the proof of the Theorem \ref{MainTheorem}. In addition, the following elementary facts play an important role in this proof (for this approach see \cite{MMSchw_Kob}).

\begin{itemize}
\item[1)] If $u:\mathbb{U}\rightarrow(0,+\infty)$ is a harmonic function, then it is known that there exists a holomorphic function $f:\mathbb{U}\rightarrow \mathbb{K}$ such that $\re{f}=u$ and $f(0)=u(0)$. 

\item[2)] Using the Schwarz-Pick inequality, we obtain
\begin{equation}\label{HalfPlane:1}
 \rho_{\mathbb{K}}(f(z))|f'(z)|\leqslant\rho_{\mathbb{U}}(z), \quad z\in \mathbb{U}.
\end{equation}

\item[3)] Due to the Cauchy-Riemann equations, we have $f'=u_x-iu_y$, i.e. $f'=(u_x,-u_y)=\overline{\nabla u}$. From this follows
\begin{equation}\label{HalfPlane2}
 |f'(z)|=|\overline{\nabla u}(z)|=|\nabla u(z)|, \quad z\in \mathbb{U}.
\end{equation}

\item[4)] Since $\displaystyle \rho_{\mathbb{K}}(z)=\frac{1}{\re{z}}$, $z\in \mathbb{K}$ we get
\begin{equation}\label{HalfPlane3}
 \rho_{\mathbb{K}}(f(z))=\rho_{\mathbb{K}}(u(z)), \quad z\in \mathbb{U}.
\end{equation}

\item[5)] From \eqref{HalfPlane:1}, \eqref{HalfPlane2} and \eqref{HalfPlane3} we immediately obtain the following proposition.
\end{itemize}

\begin{proposition}\label{Thm:SchwarzPickForPositiveHarmonic}
Let $u:\mathbb{U}\rightarrow(0,+\infty)$ be a harmonic function. Then
 \begin{equation*}
 \rho_{\mathbb{K}}(u(z))|\nabla u(z)|\leqslant\rho_{\mathbb{U}}(z),
 \end{equation*}
 i.e.
 \begin{equation}\label{Thm:SchwarzPickForPositiveHarmonic:fla1}
 |\nabla u(z)|\leqslant\frac{2u(z)}{1-|z|^2},
 \end{equation}
for all $z\in \mathbb{U}$. Moreover, if $u$ is the real part of a conformal isomorphism from $\mathbb{U}$ onto $\mathbb{K}$, then in \eqref{Thm:SchwarzPickForPositiveHarmonic:fla1} equality holds for all $z\in \mathbb{U}$. On the other hand, if in \eqref{Thm:SchwarzPickForPositiveHarmonic:fla1} equality holds for some $z\in \mathbb{U}$, then $u$ is the real part of a conformal isomorphism $\mathbb{U}$ onto $\mathbb{K}$.
\end{proposition}

We call the facts 1)-5) together with the Proposition \ref{Thm:SchwarzPickForPositiveHarmonic} the \emph{half-plane method}. This method can be seen as an analogous version of the strip method used in the paper \cite{MMandMS} (see also \cite{MMSchw_Kob}).

Note that the following theorem can be derived as a corollary of \eqref{Thm:SchwarzPickForPositiveHarmonic:fla1}.

\begin{theorem}[{\cite[Theorem 1.1]{MMar}}]\label{Thm:Markovic}
Let $u:\mathbb{U}\rightarrow(0,+\infty)$ be a harmonic function. Then
\begin{equation}\label{Thm:MarkovicIneq}
 d_{\mathbb{K}}(u(z_1),u(z_2))\leqslant d_{\mathbb{U}}(z_1,z_2),
\end{equation}
for all $z_1,z_2\in \mathbb{U}$.
\end{theorem}

\begin{remark}
In \cite{MMar}, M. Markovi\'c proved the Theorem \ref{Thm:Markovic} for the case where the domain of the function $u$ is $\mathbb{H}=\{z\in \mathbb{C}: \im{z}>0\}$, but this is not essential since $\mathbb{H}$ and $\mathbb{U}$ are conformally equivalent.
\end{remark}

\begin{remark}
Considering the proof of the Theorem \ref{Thm:Markovic} from the paper \cite{PMel}, it is straightforward to prove that \eqref{Thm:MarkovicIneq} is equivalent to Harnack's inequality.
\end{remark}

For further results of this type of inequalities, we refer the reader to \cite{MSphd} and the literature cited there.

\section{Proof of the main result}

First we recall another notion and prove two lemmas.

For a hyperbolic domain $\Omega$, $a\in\Omega$ and $r>0$ we denote by $\overline{D}_{\Omega}(a,r)$ a hyperbolic closed disc with centre $a$ and radius $r$, i.e. $\overline{D}_{\Omega}(a,r)=\{z\in\Omega: d_{\Omega}(z,a)\leqslant r\}$. It is simply checked that
\begin{equation}\label{HyperbolicDiscIsEuclideanDisc}
 \overline{D}_{\mathbb{U}}(0,2\artanh{r})=\{z\in \mathbb{U}: |z|\leqslant r\},
\end{equation}
for all $r\in(0,1)$.

\begin{lemma}\label{Lemma1}
Let $b\in(0,+\infty)$ and $r\in(0,1)$. If $z\in\overline{D}_{\mathbb{K}}(b,2\artanh{r})$ then
\begin{equation*}
    \re{z}\in\left[\frac{1-r}{1+r}b,\frac{1+r}{1-r}b\right].
\end{equation*}
\end{lemma}

\begin{proof}
Recall that the mapping $\psi$ defined by $\displaystyle \psi(z)=\frac{z-1}{z+1}$ is a conformal isomorphism from $\mathbb{K}$ onto $\mathbb{U}$. Set $\kappa=\psi^{-1}$, $a=\psi(b)$ and $\kappa_b=\kappa\circ\varphi_{a}$. It is clear that $\kappa_b$ is a conformal isomorphism from $\mathbb{U}$ onto $\mathbb{K}$ and $\kappa_b(0)=b$. Therefore, $\kappa_b$ is an isometry from the metric space $(\mathbb{U},d_\mathbb{U})$ onto the metric space $(\mathbb{K},d_\mathbb{K})$. Therefore
\begin{equation*}
 \kappa_b\Big(\overline{D}_{\mathbb{U}}(0,2\artanh{r})\Big)=\overline{D}_{\mathbb{K}}(b,2\artanh{r}).
\end{equation*}
On the other hand, taking into account \eqref{HyperbolicDiscIsEuclideanDisc} we have
\begin{equation*}
 \kappa_b\Big(\overline{D}_{\mathbb{U}}(0,2\artanh{r})\Big)=\kappa_b\Big(\{z\in \mathbb{U}: |z|\leqslant r\}\Big).
\end{equation*}
Furthermore, we obtain by straightforward calculations that
$\kappa_b\Big(\{z\in \mathbb{U}: |z|\leqslant r\}\Big)$ is a closed Euclidean disc with centre $\displaystyle c=\frac{1+r^2}{1-r^2}b$ and radius $\displaystyle R=\frac{2r}{1-r^2}b$. So, if $\displaystyle z\in\kappa_b\Big(\{z\in \mathbb{U}: |z|\leqslant r\}\Big)$ then
\begin{equation*}
 \re{z}\in[c-R,c+R]=\left[\frac{1-r}{1+r}b,\frac{1+r}{1-r}b\right].
\end{equation*}
\end{proof}

\begin{lemma}\label{Lemma2}
Let $0\leqslant c\leqslant 1$. Then, for all $z\in\mathbb{U}$,
\begin{align}
\begin{split}\label{glava2:LemmaRelatedBeardonCarneEq}
   \log(\cosh{d_{\mathbb{U}}(z,0)}+c\sinh{d_{\mathbb{U}}(z,0)}) &= \log\left(\frac{1+|z|^2+2c|z|}{1-|z|^2}\right)  \\
                                                                &=  d_{\mathbb{U}}(|z|\varphi_{c}(|z|),0).
\end{split}
\end{align}
\end{lemma}

\begin{proof}
Let $z\in\mathbb{U}$ be arbitrary. The first equality in (\ref{glava2:LemmaRelatedBeardonCarneEq}) holds due to the fact that
\begin{equation*}
 \cosh{d_{\mathbb{U}}(z,0)}=\frac{1+|z|^2}{1-|z|^2} \quad\mbox{ i }\quad \sinh{d_{\mathbb{U}}(z,0)}=\frac{2|z|}{1-|z|^2}.
\end{equation*}
To prove the second equality in (\ref{glava2:LemmaRelatedBeardonCarneEq}), we have to determine a point $R(z)\in[0,1)$ such that
\begin{equation}\label{glava2:LemmaRelatedBeardonCarne:dok1}
 d_{\mathbb{U}}(R(z),0)=\log\left(\frac{1+|z|^2+2c|z|}{1-|z|^2}\right).
\end{equation}
Note that the equality (\ref{glava2:LemmaRelatedBeardonCarne:dok1}) is equivalent to the equality
\begin{equation*}
 \frac{1+R(z)}{1-R(z)}=\frac{1+|z|^2+2c|z|}{1-|z|^2},
\end{equation*}
and from this we get $\displaystyle R(z)=|z|\frac{c+|z|}{1+c|z|}=|z|\varphi_{c}(|z|)$.
\end{proof}

\begin{proof}[Proof of the Theorem \ref{MainTheorem}]
By using the half-plane method, we obtain that there is a unique holomorphic function $f:\mathbb{U}\rightarrow\mathbb{K}$ such that
$\re{f}=u$, $f(0)=u(0)$, $|f'(0)|=|\nabla u(0)|$ and
\begin{equation*}
 |f^h(0)|=\frac{\rho_{\mathbb{K}}(f(0))}{\rho_{\mathbb{U}}(0)}|f'(0)|=\frac{|\nabla u(0)|}{2u(0)}.
\end{equation*}
Define $\displaystyle c=\frac{|\nabla u(0)|}{2u(0)}$. According to Theorem \ref{Thm:BeardonCarne}, if $\Omega_1=\mathbb{U}$ and $\Omega_2=\mathbb{K}$ we get
\begin{equation}\label{Proof:fla2}
 d_{\mathbb{K}}(f(z),f(0))\leqslant\log\Big(\cosh{d_{\mathbb{U}}(z,0)}+c\sinh{d_{\mathbb{U}}(z,0)}\Big)
\end{equation}
for all $z\in \mathbb{U}$. By Lemma \ref{Lemma2} we thus have
\begin{equation*}
 d_{\mathbb{K}}(f(z),f(0))\leqslant 2\artanh{\Big(|z|\varphi_{c}(|z|)\Big)},
\end{equation*}
i.e
\begin{equation*}
 f(z)\in \overline{D}_{\mathbb{K}}\Big(f(0),2\artanh{\Big(|z|\varphi_{c}(|z|)\Big)}\Big).
\end{equation*}
Using Lemma \ref{Lemma1} (note that $f(0)=u(0)$ and $u=\re{f}$) we now obtain
\begin{equation}\label{Proof:fla3}
 u(z)\in\left[\frac{1-|z|\varphi_c(|z|)}{1+|z|\varphi_c(|z|)}u(0),\frac{1+|z|\varphi_c(|z|)}{1-|z|\varphi_c(|z|)}u(0)\right]
\end{equation}
But since $\displaystyle \frac{1+|z|\varphi_c(|z|)}{1-|z|\varphi_c(|z|)}=\frac{1+|z|^2+2c|z|}{1-|z|^2}$, it follows that \eqref{Proof:fla3} is equivalent to \eqref{HarnackIneqTheoremIneq}.
\end{proof}

\end{document}